\documentclass{amsart} 
\usepackage{macros}
\standardsettings

\newcommand{\Length}{\mathrm{length}}
\newcommand{\reg}{\mathrm{reg}}

\draftfalse

\title[Dimensions of very well intrinsically approximable sets]{Hausdorff dimensions of very well intrinsically approximable subsets of quadratic hypersurfaces}
\date{}

\begin{document}

\authorlior\authorkeith\authordavid

\begin{abstract}

We prove an analogue of a theorem of A. Pollington and S. Velani ('05), furnishing an upper bound on the Hausdorff dimension of certain subsets of the set of very well intrinsically approximable points on a quadratic hypersurface. The proof incorporates the framework of intrinsic approximation on such hypersurfaces first developed in the authors' joint work with D. Kleinbock (preprint '14) with ideas from work of D. Kleinbock, E. Lindenstrauss, and B. Weiss ('04).

\end{abstract}
\maketitle

\section{Introduction}

In its classical form, the field of Diophantine approximation investigates the quality by which an irrational vector $\xx \in \R^d$ can be approximated by rationals $\frac{\pp}{q} \in \Q^d$.\footnote{Throughout the paper, rationals will be written as $\frac{\pp}{q}$, where $\pp$ is a primitive integer vector and $q$ is a positive integer.} One of the most significant results is Dirichlet's theorem, a corollary of which states that for every irrational $\xx$, there exist infinitely many reduced rationals $\frac{\pp}{q}$ satisfying
\[
\left\|\xx - \frac{\pp}{q}\right\| < \frac1{q^{1+1/d}},
\]
where $\left\| \cdot \right\|$ denotes the max norm. Motivated by this result, we recall the notion of the \emph{exponent of irrationality} of $\xx$:
\[
\omega(\xx) := \sup\left\{c > 0: \exists^\infty \frac{\pp}{q} \text{ satisfying } \left\| \xx - \frac{\pp}{q} \right\| < \frac1{q^c}\right\}.
\]
Here and hereafter the notation $\exists^\infty$ stands for ``there exist infinitely many.''

Clearly, Dirichlet's corollary implies that $\omega(\xx) \geq 1 + \frac1{d}$ for all $\xx \in \R^d$. Let
\[
W_c := \{\xx \in \R^d : \omega(\xx) > c\}.
\]
We call a vector $\xx$ \emph{very well approximable}, denoted $\xx \in \VWA$, if $\omega(\xx) > 1 + \frac1{d}$, i.e. if
\[
\xx \in \bigcup_{c > 1 + \frac1{d}} W_c.
\]
It is well known that the set $\VWA$ is a Lebesgue null set of full Hausdorff dimension. More precisely, a result of V. Jarn\'ik states that $\HD(W_c) = \frac{d+1}{c}$ for all $c \geq 1 + \frac1d$. Here and hereafter $\HD$ stands for Hausdorff dimension.

In \cite{FKMS}, D. Kleinbock and the authors developed a theory of \emph{intrinsic approximation on quadratic hypersurfaces} which shares many features with the classical theory. Fix $d\geq 2$, let $P:\R^d\to\R$ be a quadratic polynomial with integral coefficients, and let $Z_P$ denote the zero set of $P$. Then intrinsic approximation on the quadratic hypersurface $Z_P$ is the theory of approximating points $\xx\in Z_P$ by rational points $\frac\pp q\in \Q^d\cap Z_P$.\footnote{In \cite{FKMS}, most of the results were phrased in terms of the projectivization of $Z_P$, which can be described in terms of the light-cone of a quadratic form $Q$ depending on $P$. For the present paper, it is easier to work in the affine setting, for which purpose we can use the Affine Corollaries provided in \cite{FKMS}. The interested reader may verify that the results of this paper can be translated back into the projective setting if desired.} The word ``intrinsic'' refers to the fact that the rational points are required to lie in $Z_P$, rather than just the point $\xx$. If this requirement is omitted, what results is the classical theory of \emph{Diophantine approximation on manifolds}: cf. \cite{Beresnevich_Khinchin, Beresnevich_BA, KleinbockMargulis2}.


One of the main theorems in \cite{FKMS} is a Dirichlet-type theorem for intrinsic approximation:

\begin{theorem}[{\cite[Theorem 8.1, Affine Corollary]{FKMS}} + {\cite[Remark 8.7]{FKMS}}]
\label{theoremFKMS}
Let $P:\R^{d} \to \R$ be a quadratic polynomial with integral coefficients, and let $Z = Z_P$ denote its zero set. Suppose that $\Q^d\cap Z_\reg\neq\emptyset$, where $Z_\reg$ is the set of points at which $Z$ is smooth (i.e. at which $\nabla P\neq \0$). Then for every $\xx \in Z$, there exists a constant $C = C(\xx) > 0$ and infinitely many $\frac{\pp}{q} \in Z \cap \Q^{d}$ satisfying
\[
\left\| \xx - \frac{\pp}{q} \right\| \leq \frac{C}{q}\cdot
\]
\end{theorem}


Let us define a function $\omega_Z \colon Z \to \R$ via the formula
\[
\omega_Z(\xx) = \sup\left\{ c > 0 : \exists C =C(\xx) \; \exists^\infty \frac{\pp}{q} \in Z \cap \Q^{d} \text{ with } \left\| \xx -\frac{\pp}{q} \right\| \leq \frac{C}{q^c} \right\}.
\]
Then by Theorem \ref{theoremFKMS}, we have $\omega_Z(\xx) \geq 1$ for all $\xx\in Z$. We define $W_{c,Z} := \{\xx \in Z : \omega_Z(\xx) > c\}$ and
\[
\VWA_Z := \bigcup_{c > 1} W_{c,Z}.
\]
It is shown in \cite[Theorem 5.5]{FKMS} that $\VWA_Z$ is a Lebesgue nullset and moreover \cite[Theorem 2.13]{FKMS} that for all $c > 1$, $\HD(W_{c,Z}) = \frac{d - 1}{c}$.

In \cite[Theorem 5.5]{FKMS}, the nullity of $\VWA_Z$ is also shown to hold for measures on $Z$ which are \emph{absolutely friendly} (see \6\ref{sectiondefinitions} for the definition), a class of measures implicitly introduced in \cite{KLW} and explicitly formalized in \cite{PollingtonVelani}, and which has since played a preeminent role in Diophantine approximation and metric measure theory. Historically, these measures are motivated by the following question: if we take a `nice' subset $S \subset \R^d$ of strictly smaller dimension, should we expect the approximability of points in $S$ to reflect the approximation properties of points in $\R^d$? For example, since $\VWA$ is a Lebesgue null set, when should we expect that almost no point of $S$ is very well approximated by rationals, with respect to some natural measure on $S$? It is shown in \cite{KLW} that if $\mu$ is an (absolutely) friendly measure, then $\mu$-a.e. point in its support is not very well approximable,\footnote{The results of \cite{KLW} hold in the stronger multiplicative framework as well.} i.e. $\mu(\VWA) = \mu(\VWA \cap \Supp(\mu)) = 0$. Equivalently, $\mu(W_c) = 0$ for all $c\geq 1 + \frac1d$.

Motivated by this new definition, A. Pollington and S. Velani established \cite{PollingtonVelani} analogues of some classical results in Diophantine approximation for absolutely friendly measures, amongst other things proving results on the Hausdorff dimension of $\VWA \cap \Supp(\mu)$. One of their results is the following (see \cite[Corollary 2]{PollingtonVelani}): if a measure $\mu$ on $\R^d$ is absolutely $\alpha$-friendly and Ahlfors $\delta$-regular, then for any $c \geq 1 + \frac1{d}$,
\begin{equation}
\label{pollingtonvelani}
\HD(W_c \cap \Supp(\mu)) \leq \delta - \alpha\left(1 - \frac{1+\frac1d}{c}\right).
\end{equation}
In this paper, we extend Pollington and Velani's result to the setting of intrinsic approximation on quadratic hypersrufaces. Namely, for a large class of measures $\mu$ we obtain an estimate on the Hausdorff dimension of the sets $W_{c,Z}$ similar to \eqref{pollingtonvelani}:
\begin{theorem}\label{maintheorem}
Let $\mu$ be an absolutely $\alpha$-friendly and Ahlfors $\delta$-regular on $\R^{d - 1}$, let $U\subset\R^{d - 1}$ be an open set containing $\Supp(\mu)$, let $\Psi \colon U \to Z_\reg$ be a local coordinate chart, and let $\nu := \Psi[\mu]$. Then for any $c \geq 1$,
\begin{equation}
\label{main}
\HD(W_{c,Z} \cap \Supp(\nu)) \leq \delta - \alpha\left(1 - \frac{1}{c}\right).
\end{equation}
\end{theorem}

We prove Theorem \ref{maintheorem} by adapting the methods of \cite{PollingtonVelani} to our setting. We remark that the main difficulty in translating their proof occurs when we need to estimate the measure of a set which, in the setting of \cite{PollingtonVelani}, is a neighborhood of a hyperplane (and thus its measure can be directly estimated using the definition of absolute friendliness), but in our setting the corresponding set is not the neighborhood of a hyperplane, but rather the neighborhood of a quadratic hypersurface. See the beginning of \6\ref{sectionproof} for more details.

It follows from the Mass Distribution Principle, Lemma \ref{mdp} below, that if $\mu$ is an Ahlfors $\delta$-regular measure, then any set with positive $\mu$-measure must have Hausdorff dimension at least $\delta$. Therefore Theorem \ref{maintheorem} result generalizes \cite[Theorem 5.5]{FKMS}, which asserts that the sets $W_{c,Z}$ and $\VWA_Z$ are $\mu$-null sets.

It is worth mentioning that an upper bound is the best possible result in this situation, for the following reason: it follows from \cite[Proposition D.1]{FSU4} and \cite[Theorem 5.3]{FKMS} that the set of intrinsically \emph{badly approximable} points on $Z$ supports Ahlfors $\delta$-regular measures with $\delta$ arbitrarily close to $k = \dim(Z)$; by \cite[Proposition 6.3]{KleinbockWeiss1} these measures are absolutely friendly once $\delta > k - 1$.  Since no badly approximable point is very well approximable, it follows that for any such measure $\mu$, the dimension of $\VWA_Z \cap \Supp(\mu)$ is zero, since this intersection is empty.

{\bf Acknowledgements.} The first-named author was supported in part by the Simons Foundation grant \#245708.

\section{Definitions}
\label{sectiondefinitions}

\begin{convention*}
\label{conventionimplied}
The symbols $\lesssim_\times$, $\gtrsim_\times$, and $\asymp_\times$ will denote coarse multiplicative asymptotics. For example, $A\lesssim_\times B$ means that there exists a constant $C > 0$ (the \emph{implied constant}) such that $A\leq C B$. It is understood that the implied constant $C$ is only allowed to depend on certain ``universal'' parameters, to be understood from context.
\end{convention*}

\subsection{Absolutely friendly and Ahlfors regular measures}

We start with a definition introduced in \cite{KLW}: if $\mu$ is a locally finite Borel measure on $\R^k$ and $\alpha>0$, one says that $\mu$ is \emph{absolutely $\alpha$-decaying\footnote{This terminology differs slightly from the one in \cite{KLW}, where a less uniform version was considered.}\/} if there exists $\rho_0 > 0$ such that for every $0 < \rho \leq \rho_0$, every $\xx \in\Supp\,\mu$, every affine hyperplane $\LL\subset \R^k$, and every $\epsilon > 0$, one has
\[
\mu\big(B(\xx,\rho)\cap \thickvar\LL{\epsilon\rho}\big) \lesssim_\times \epsilon^{\alpha}\mu\big(B(\xx,\rho)\big).
\]
Here $\thickvar\LL\rho$ denotes the closed $\rho$-thickening of the affine hyperplane $\LL$, i.e. $\thickvar\LL\rho = \left\{ \xx \in \R^d : \operatorname{dist}(\xx, \LL) \leq \rho\right\}$.

Another useful property is the so-called Federer (doubling) condition. One says that $\mu$ is \emph{Federer\/} if there exists $\rho_0 >0$ such that
\[
\mu\big(B(\xx ,2\rho)\big) \asymp_\times \mu\big(B(\xx,\rho)\big)\,\quad\forall\, \xx\in \Supp(\mu)\; \forall 0 < \rho\leq\rho_0.
\]
Measures which are both absolutely $\alpha$-decaying and Federer are called \emph{absolutely $\alpha$-friendly\/}, a term coined in \cite{PollingtonVelani}.

Many examples of absolutely friendly measures can be found in \cite{KLW, KleinbockWeiss1, StratmannUrbanski1, Urbanski}. The Federer condition is very well studied; it obviously holds when $\mu$ is \emph{Ahlfors regular\/}, i.e.\ when there exist $\delta,\rho_0 > 0$ such that
\[
\mu\big(B(\xx,\rho)\big) \asymp_\times \rho^{\delta}\,\quad\forall\,\xx\in \Supp(\mu)\; \forall\,0 < \rho\leq\rho_0 .
\]
The above property for a fixed $\delta$ will be referred to as \emph{Ahlfors $\delta$-regularity\/}. It is easy to see that the Hausdorff dimension\ of the support of a Ahlfors $\delta$-regular measure is equal to $\delta$. An important class of examples of absolutely decaying and Ahlfors regular measures is provided by limit measures of irreducible families of contracting similarities \cite{KLW} or conformal transformations \cite{Urbanski} of $\R^k$ satisfying the open set condition, as defined in \cite{Hutchinson}. See however \cite{KleinbockWeiss1} for an example of an absolutely friendly measure which is not Ahlfors regular.

The following is a well-known and useful consequence of Ahlfors $\delta$-regularity:
\begin{lemma}[{\cite[Lemma 4.2]{Falconer_book}}]
\label{mdp}
If $\mu(B(\xx, \rho)) \lesssim_\times \rho^\delta$ for all $\xx \in \Supp(\mu)$ and all $0 < \rho \leq \rho_0$, then whenever a set $E$ satisfies $\mu(E) > 0$, we have $\HD(E) \geq \delta$.
\end{lemma}

\subsection{Hausdorff--Cantelli Lemma}
Although we do not define Hausdorff dimension here (see \cite{Falconer_book} for a definition), we will recall the main tool with which we will bound the Hausdorff dimension from above.

\begin{lemma}[{\cite[Lemma 3.9]{BernikDodson}}]
\label{HClemma}
Let $E$ be a subset of $\R^d$ and suppose that
\[
E \subseteq \{\xx \in \R^d : \xx \in H \text{ for infinitely many } H\in\CC\},
\]
where $\CC$ is a family of open balls. If for some $s >0$,
\[
\operatorname{cost}_s(\CC) := \displaystyle\sum_{H\in\CC} \diam^s(H) < \infty,
\]
then $\HD(E) \leq s$.
\end{lemma}

\section{Proof of the Main Result}
\label{sectionproof}

\subsection{Strategy of the proof}
A crucial ingredient in the proof of \cite[Corollary 2]{PollingtonVelani} was the \emph{simplex lemma} \cite[Lemma 4]{KTV}, which states that there exists $\kappa > 0$ such that for every ball $B = B(\xx,\rho)\subset\R^d$, the set
\[
\left\{\frac\pp q\in\Q^d : q \leq \kappa \rho^{-1/(1 + \frac1d)}\right\}
\]
is contained in an affine hyperplane $\LL = \LL_B\subset\R^d$. %
In \cite{FKMS}, we introduced an analogue of the simplex lemma for intrinsic approximation on nonsingular manifolds:
\begin{lemma}[{\cite[Lemma 5.1]{FKMS}}]
\label{lemmasimplex}
Let $M\subset\R^d$ be a submanifold of dimension $k$, let $U\subset\R^k$ be an open set, let $\Psi:U\to M$ be a local parameterization, and let $V\subset U$ be compact. Then there exists $\kappa > 0$ such that for every ball $B = B(\xx,\rho)\subset V$, the set
\[
\left\{\frac\pp q \in \Q^d\cap \Psi(B(\xx,\rho)) : q \leq \kappa \rho^{-1/c(k,d)}\right\}
\]
is contained in an affine hyperplane $\LL = \LL_{B,\Psi}\subset\R^d$. Here $c(k,d) > 0$ is a constant depending on $k$ and $d$, with $c(d - 1,d) = 1$.
\end{lemma}

Now if $\mu$ is an absolutely $\alpha$-friendly measure, then we can estimate the measure of the set $B\cap \thickvar{\LL_B}{\epsilon\rho}$ by using the definition of absolute friendliness. This estimate is needed in the proof of \cite[Corollary 2]{PollingtonVelani}. In the setting of Lemma \ref{lemmasimplex}, we instead need to estimate the measure of the set $B\cap \thickvar{\Psi^{-1}(\LL_{B,\Psi})}{\epsilon\rho}$, and we cannot do this directly from the definition of absolute friendliness. However, if the manifold $M$ is a quadratic hypersurface,\footnote{A reviewer asked what happens if $M$ is not a quadratic hypersurface, but merely a smooth manifold. In this case, the set $\Psi^{-1}(\LL_{B,\Psi})$ (which is not necessarily a manifold) may have complicated singularities, and it is not clear how to analyze them in such generality. By contrast, if $M$ is a quadratic hypersurface then so is $\Psi^{-1}(\LL_{B,\Psi})$, and this gives us some control over its singularities.} then if the chart $\Psi$ is chosen appropriately, then $\Psi^{-1}(\LL_{B,\Psi})$ will be a quadratic hypersurface in $\R^{d - 1}$, and it turns out that we can use this fact to bound the measure appropriately. Another difficulty comes from the fact that the choice of $\Psi$ such that $\Psi^{-1}(\LL_{B,\Psi})$ is a quadratic hypersurface may not be the same as the $\Psi$ which is given in the hypothesis of Theorem \ref{maintheorem}. So the proof of Theorem \ref{maintheorem} has three major parts:
\begin{itemize}
\item[1.] proving that absolute friendliness is preserved under nonsingular transformations;
\item[2.] estimating the measure of a neighborhood of a quadratic hypersurface under an absolutely friendly measure;
\item[3.] using the proof idea of \cite{PollingtonVelani} to finish the proof.
\end{itemize}
In what follows, $\mu$ always denotes a compactly supported absolutely $\alpha$-friendly measure on $\R^k$, where $k = d - 1$. Also, we fix $\beta < \alpha$.

\subsection{Absolute friendliness and nonsingular transformations}
We shall prove the following:
\begin{proposition}
\label{propositionpreserved}
Let $U\subset\R^k$ be an open set containing $\Supp(\mu)$, and let $\Phi:U\to V \subset \R^k$ be an invertible nonsingular transformation. Then $\Phi[\mu]$ is absolutely $\beta$-friendly.
\end{proposition}
The following lemma will be proved in somewhat greater generality than is needed to prove Proposition \ref{propositionpreserved}, since this generality will be used later in the proof of Theorem \ref{maintheorem}. 

\begin{definition}
A set $S\subset\R^{k - 1}$ is \emph{$K$-quasiconvex} if for every $\xx,\yy\in S$ there exists a piecewise smooth path $\gamma\subset S$ connecting $\xx$ and $\yy$ such that $\Length(\gamma)\leq K\|\yy - \xx\|$.
\end{definition}

\begin{lemma}
\label{lemma0}
Let $S\subset\R^{k - 1}$ be a $K_1$-quasiconvex set, and let $f:S\to\R$. Let $\Gamma = \Gamma(S,f)$ denote the graph of $f$, so that $\Gamma\subset\R^k$. Then for every ball $B = B(\xx,\rho)$ with $\xx\in\Supp(\mu)$ and $0 < \rho \leq K_2/\|f''\|_S$ and for all $\epsilon > 0$, we have
\begin{align} \label{2B}
\mu\big(\thickvar{B\cap\Gamma}{\epsilon\rho}\big) &\lesssim_\times \epsilon^{\alpha/2} \mu(\thickvar{B\cap\Gamma}{\sqrt\epsilon\rho}) \note{if $\epsilon\leq 1$}\\ \label{betabound}
\mu\big(B\cap \thickvar{\Gamma}{\epsilon\rho}\big) &\lesssim_\times \epsilon^\beta \mu(B).
\end{align}
The implied constants can depend on $K_1$, $K_2$, $\mu$, and $\beta$.
\end{lemma}
\begin{proof}[Proof of \eqref{2B}]
If $\epsilon\geq 1/4$ then \eqref{2B} is trivial, so suppose that $\epsilon\leq 1/4$ and thus $\sqrt\epsilon \geq 2\epsilon$.

Let $\w\rho = \sqrt\epsilon\rho/2$, and let $A\subset \thickvar{B\cap\Gamma}{\epsilon\rho}$ be a maximal $\w\rho$-separated set. Fix $\yy\in A$ and let $\w B = B(\yy,\w\rho)$. Fix $\zz\in \w B$ and let $\gamma\subset S$ be a piecewise smooth path connecting $\yy$ and $\zz$ such that $\Length(\gamma)\leq K_1 \|\zz - \yy\|$. Applying the fundamental theorem of calculus to the functions $f'$ and $f$ on $\gamma$ gives
\begin{align*}
\|f'(\cdot) - f'(\yy)\|_\gamma &\leq \|f''\|_S \cdot \Length(\gamma) \leq (K_2/\rho) (K_1 \w\rho) \leq K_1 K_2 \sqrt\epsilon\\
\|f(\zz) - f(\yy) - f'(\yy)[\zz - \yy]\| &\leq \|f'(\cdot) - f'(\yy)\|_\gamma \cdot \Length(\gamma) \leq (K_1 K_2 \sqrt\epsilon) (K_1 \w\rho) = K_1^2 K_2 \sqrt\epsilon\w\rho.
\end{align*}
So if $\LL$ denotes the hyperplane $\{(\zz,f(\yy) + f'(\yy)[\zz - \yy]) : \zz\in\R^{k - 1}\}$, then $B\cap\Gamma\subset \thickvar\LL{K_1^2 K_2 \sqrt\epsilon\w\rho}$. Thus
\begin{align*}
\mu\big(\w B\cap \thickvar{B\cap\Gamma}{\epsilon\rho}\big)
&=_\pt \mu\big(\w B\cap \thickvar\LL{(K_1^2 K_2 + 2) \sqrt\epsilon\w\rho}\big) \since{$\epsilon\rho = 2\sqrt\epsilon\w\rho$}\\
&\lesssim_\times (\sqrt\epsilon)^\alpha \mu(\w B) = \epsilon^{\alpha/2} \mu(\w B) \note{$\alpha$-decay property}\\
\mu\big(\thickvar{B\cap\Gamma}{\epsilon\rho}\big)
&\leq_\pt \sum_{\yy\in A} \mu\big(B(\yy,\w\rho)\cap \thickvar{B\cap\Gamma}{\epsilon\rho}\big)\\
&\lesssim_\times \sum_{\yy\in A} \epsilon^{\alpha/2} \mu\big(B(\yy,\w\rho)\big)\\
&\lesssim_\times  \epsilon^{\alpha/2} \mu\big(\thickvar{B\cap\Gamma}{\epsilon\rho + \w\rho}\big) \note{bounded multiplicity}\\
&\leq_\pt \epsilon^{\alpha/2} \mu\big(\thickvar{B\cap\Gamma}{\sqrt\epsilon\rho}\big) \since{$\sqrt\epsilon\geq 2\epsilon$}
&\qedhere\end{align*}
\end{proof}
\begin{proof}[Proof of \eqref{betabound}]
Let $C > 1$ denote the implied constant of \eqref{2B}. Let $N = \lfloor \log_2 \log(1/\epsilon)\rfloor$, and for each $n = 0,\ldots,N - 1$, plug in $\epsilon := \epsilon^{1/2^n}$ and $\rho := 2\rho$ in \eqref{2B}. (If $N$ is negative or undefined, then $\epsilon\geq 1/e$ and thus \eqref{betabound} is trivially true.) Taking the product yields
\[
\frac{\mu\big(\thickvar{2B\cap\Gamma}{\epsilon\rho}\big)}{\mu\big(\thickvar{2B\cap\Gamma}{\epsilon^{1/2^N} \rho}\big)} \leq \prod_{n = 0}^{N - 1} C \epsilon^{\alpha/2^{n + 1}} = C^N \epsilon^{\alpha(1 - 2^{-N})}.
\]
Since $\epsilon^{1/2^N}\asymp_\times 1$ and $\mu\big(\thickvar{2B\cap\Gamma}{\epsilon^{1/2^N} \rho}\big) \leq \mu(3B)\asymp_\times \mu(B)$, we get
\[
\mu\big(B\cap \thickvar{\Gamma}{\epsilon\rho}\big) \leq \mu\big(\thickvar{2B\cap \Gamma}{\epsilon\rho}\big) \lesssim_\times C^{\log_2\log(1/\epsilon)} \epsilon^\alpha \mu(B).
\]
Since $\beta < \alpha$, this demonstrates \eqref{betabound}.
\end{proof}

Now let $U\subset\R^k$ be an open set containing $\Supp(\mu)$ and let $\Phi:U\to V \subset \R^k$ be an invertible nonsingular transformation. Since $\Supp(\mu)$ is compact, after shrinking $U$ we can assume that $\|\Phi'(\xx)^{-1}\| \leq c_0$ and $\|\Phi''(\xx)\| \leq c_1$ for all $\xx\in U$. Fix a ball $B(\xx,\rho)$ centered at $\xx\in\Supp(\mu)$, let $\LL\subset\R^k$ be an affine hyperplane, and let $M = \Phi^{-1}(\LL)$. Writing $\LL = L^{-1}(t)$ for some linear map $L:\R^k\to\R$ and some $t\in\R$, we get $M = (L\circ\Phi)^{-1}(t)$. Write $L[\yy] = \lb \yy,\ww\rb$ for some $\ww\in\R^k$, and without loss of generality suppose that $\|\ww\| = 1$. Let $\vv\in\R^k$ be the unique unit vector such that $\Phi'(\xx)[\vv] = a\ww$ for some $a > 0$. Then
\[
(L\circ\Phi)'(\xx)[\vv] = a = \|\Phi'(\xx)[\vv]\| \geq 1/\|\Phi'(\xx)^{-1}\| \geq 1/c_0 > 0.
\]
So for all $\yy\in B(\xx,\rho)$, we have $(L\circ\Phi)'(\yy)[\vv] \geq 1/c_0 - c_1\rho$. If we assume that $\rho \leq \rho_0 := 1/(2c_0 c_1)$, then we get $(L\circ\Phi)'(\yy)[\vv] \geq 1/(2c_0)$, so by the implicit function theorem, $B(\xx,\rho)\cap M$ is contained in a rotated version of a set $\Gamma$ as in Lemma \ref{lemma0}. Moreover, the constants $K_1$ and $K_2$ corresponding to this $\Gamma$ will be uniform with respect to $\xx$. So by Lemma \ref{lemma0} we get
\[
\mu\big(B(\xx,\rho)\cap\thickvar M{\epsilon\rho}\big) \lesssim_\times \epsilon^\beta \mu(B(\xx,\rho)).
\]
Since $\Phi$ is bi-Lipschitz (after shrinking $U$), this inequality implies that the pushforward measure $\Phi[\mu]$ is absolutely $\beta$-friendly.

\subsection{Measuring neighborhoods of quadratic hypersurfaces}

\begin{lemma}
\label{lemma1}
For every quadratic hypersurface $Z\subset\R^k$, for every ball $B = B(\ww,\rho)\subset\R^k$ with $\ww\in\Supp(\mu)$ and $0 < \rho \leq 1$, and for every $\epsilon > 0$,
\[
\mu(B\cap \thickvar Z{\epsilon\rho}) \lesssim_\times \epsilon^\beta \mu(B).
\]
\end{lemma}
We emphasize that the implied constant here must be independent of $Z$.
\begin{proof}
We assume that $k\geq 2$, as otherwise the conclusion is trivially true. Write $Z = Z_P$ for some quadratic polynomial $P:\R^k\to\R$. By a perturbation argument, we can without loss of generality assume that the quadratic part of $P$ is nondegenerate. Then by applying a translation, we can without loss of generality assume that $P$ has no linear part. Finally, by applying a rotation we can without loss of generality suppose that
\[
P(\xx) = \sum_{i = 1}^k c_i x_i^2 - b
\]
for some $c_1,\ldots,c_k \neq 0$ and $b\in\R$. This is because every symmetric matrix can be diagonalized by an orthogonal matrix.

Let
\begin{align*}
\w Z &= \{\xx\in Z\cap\Rplus^k : |c_i x_i| \leq |c_k x_k| \all i = 1,\ldots,k - 1\}\\
&= \{\xx\in Z\cap\Rplus^k : c_i^2 x_i^2 \leq c_k^2 x_k^2 \all i = 1,\ldots,k - 1\}.
\end{align*}
Then $Z = G(\w Z)$, where $G$ is the finite group of signed permutation matrices. So to complete the proof, it suffices to show that for every ball $B = B(\ww,\rho)\subset\R^k$ with $\ww\in\Supp(\mu)$ and $0 < \rho \leq 1$ and for every $\epsilon > 0$,
\begin{equation}
\label{ETS}
\mu(B\cap \thickvar{\w Z}{\epsilon\rho}) \lesssim_\times \epsilon^\beta \mu(B).
\end{equation}
Let $f:\R^{k - 1}\to\R$ be the unique positive solution to $P(\xx,f(\xx)) = 0$, i.e.
\[
f(\xx) = \sqrt{\frac{b}{c_k} - \sum_{i = 1}^{k - 1} \frac{c_i}{c_k} x_i^2},
\]
and let
\begin{align*}
S &= \left\{\xx\in\R^{k - 1} : (\xx,f(\xx)) \in \w Z\right\}
= \left\{\xx\in\Rplus^{k - 1} : c_k b - c_k \sum_{i = 1}^{k - 1} c_i x_i^2 \geq c_j^2 x_j^2 \all j = 1,\ldots,k - 1\right\}.
\end{align*}
Then $\w Z = \Gamma(S,f)$ in the notation of Lemma \ref{lemma0}. We compute $f'$ and $f''$, using the notation $x_k = f(\xx)$:
\begin{align*}
f'(\xx)_i &= -\frac{c_i x_i}{c_k x_k}\\
|f'(\xx)_i| &\leq 1 \;\; (\xx\in S)\\
f''(\xx)_{ij} &= \frac{c_i x_i}{c_k x_k^2} f'(\xx)_j - \frac{c_i \delta_{ij}}{c_k x_k}
= -\left[\frac{c_i x_i c_j x_j}{c_k^2 x_k^3} + \frac{c_i \delta_{ij}}{c_k x_k}\right]\\
|f''(\xx)_{ij}| &\leq \frac{1}{x_k} + \frac{1}{x_i} \leq \frac2{\min_{\ell = 1}^k x_\ell} \;\; (\xx\in S).
\end{align*}
Since our bound for $\|f''(\xx)\|$ depends on $\xx$, we need a stronger property than just the quasiconvexity of $S$:

\begin{claim}
If $R = \prod_{i = 1}^k [a_i,b_i] \subset\Rplus^k$ is a coordinate-parallel rectangle, then the set
\[
S_R = \{\xx\in\R^{k - 1} : (\xx,f(\xx)) \in \w Z\cap R\}
\]
is $\sqrt{k - 1}$-quasiconvex.
\end{claim}
\begin{subproof}
Let $\phi:\Rplus^{k - 1}\to\Rplus^{k - 1}$ be defined by the equation
\[
\phi(\xx) = (x_1^2,\ldots,x_{k - 1}^2).
\]
Then $S_R = \phi^{-1}(T_R)$, where
\begin{align*}
T_R &= \left\{\yy\in \prod_{i = 1}^{k - 1} [a_i^2,b_i^2] : c_k b - c_k \sum_{i = 1}^{k - 1} c_i y_i \geq c_j^2 y_j \all j = 1,\ldots,k - 1, \;\; \frac{b}{c_k} - \sum_{i = 1}^{k - 1} \frac{c_i}{c_k} y_i \in [a_k^2,b_k^2]\right\}.
\end{align*}
Fix $\xx^{(0)},\xx^{(1)}\in S_R$, and let $\yy^{(0)} = \phi(\xx^{(0)})$, $\yy^{(1)} = \phi(\xx^{(1)})$. Since $T_R$ is convex, for each $t\in[0,1]$
\[
\yy^{(t)} := \yy^{(0)} + t(\yy^{(1)} - \yy^{(0)}) \in T_R
\]
and thus
\[
\xx^{(t)} := \phi^{-1}(\yy^{(t)}) \in S_R.
\]
Let $\gamma$ be the path $t\mapsto \xx^{(t)}$. Since for each $i = 1,\ldots,k - 1$ the map $t\mapsto x^{(t)}_i$ is monotonic, we have
\[
\Length(\gamma) = \int_0^1 \left\|\frac{\del\xx^{(t)}}{\del t}\right\| \; \dee t \leq \sum_{i = 1}^{k - 1} \int_0^1 \left|\frac{\del x^{(t)}_i}{\del t}\right| \; \dee t = \sum_{i = 1}^{k - 1} |x^{(1)}_i - x^{(0)}_i| \leq \sqrt{k - 1} \|\xx^{(1)} - \xx^{(0)}\|,
\]
which completes the proof.
\end{subproof}

Now let $B = B(\ww,\rho)$ be a ball with $\ww\in\Supp(\mu)$ and $0 < \rho\leq 1$, and fix $\epsilon > 0$. Fix $n\in\N$ and let
\[
\rho_n = 2^{-n} \rho, \;\; R_n = \left\{\xx\in \Rplus^k : \rho_{n + 1} \leq \min_{i = 1}^k x_i \leq \rho_n\right\}, \;\; S_n = S_{R_n}.
\]
Since $R_n$ can be written as the union of $k$ different coordinate-parallel rectangles, $S_n$ can be written as the union of $k$ different $\sqrt{k - 1}$-quasiconvex sets. Thus since
\[
\|f''\|_{S_n} \lesssim_\times 1/\rho_n,
\]
the hypotheses of Lemma \ref{lemma0} are satisfied for balls of radius $\rho_n$. Now let $A$ be a maximal $\rho_n$-separated subset of $B(\ww,\rho)\cap R_n\cap\Supp(\mu)$, so that
\begin{align*}
\mu\big(B\cap \thickvar{\Gamma(S_n,f)}{\epsilon\rho}\big)
&\leq_\pt \sum_{\xx\in A} \mu\big(B(\xx,\rho_n)\cap \thickvar{\Gamma(S_n,f)}{\epsilon\rho}\big)\\
&\lesssim_\times \sum_{\xx\in A} (2^n\epsilon)^\beta \mu \big(B(\xx,\rho_n)\big) \note{Lemma \ref{lemma0}}\\
&\lesssim_\times (2^n\epsilon)^\beta \mu\big(\thickvar{B\cap R_n}{\rho_n}\big). \note{bounded multiplicity}
\end{align*}
Let $V$ denote the union of the coordinate hyperplanes in $\R^k$, so that $R_n \subset \thickvar V{\rho_n}$. Then
\begin{align*}
\mu\big(B\cap \thickvar{\Gamma(S_n,f)}{\epsilon\rho}\big)
&\leq_\pt (2^n\epsilon)^\beta \mu\big(2B\cap \thickvar V{2\rho_n}\big)\\
&\lesssim_\times (2^n\epsilon)^\beta (2^{-n})^\alpha \mu(2B) \note{$\alpha$-decaying property}\\
&\asymp_\times 2^{-n(\alpha - \beta)} \epsilon^\beta \mu(B). \note{doubling property}\\
\mu(B\cap\thickvar{\Gamma(S,f)}{\epsilon\rho})
&\leq_\pt \sum_{n = 0}^\infty \mu\big(B\cap \thickvar{\Gamma(S_n,f)}{\epsilon\rho}\big)\\
&\lesssim_\times \epsilon^\beta \mu(B) \sum_{n = 0}^\infty 2^{-n(\alpha - \beta)} \asymp_\times \epsilon^\beta \mu(B),
\end{align*}
demonstrating \eqref{ETS}.
\end{proof}

Recall that $d = k + 1$.

\begin{corollary}
\label{corollaryquadraticneighborhood}
Let $U\subset\R^k$ be an open set containing $\Supp(\mu)$, let $\Psi:U\to Z_\reg$ be a local coordinate chart on a quadratic hypersurface $Z\subset\R^d$, and let $\nu = \Psi[\mu]$. Then for every affine hyperplane $\LL$, for every ball $B = B(\xx,r) \subset \R^d$, and for every $\epsilon > 0$,
\[
\nu\big(B\cap \thickvar{\LL\cap Z}{\epsilon\rho}\big) \lesssim_\times \epsilon^\beta \nu(B).
\]
\end{corollary}
In this corollary, the implied constant can depend on $Z$.
\begin{proof}
By Proposition \ref{propositionpreserved}, it suffices to show that there exists a covering of $Z_\reg$ by coordinate charts with the required property. We show that coordinate charts whose inverses are linear have the property. Indeed, suppose that $L\circ\Psi(\xx) = \xx$, where $L:\R^d\to\R^k$ is some linear map. Let $\LL\subset\R^d$ be an affine hyperplane. Then $\LL\cap Z$ is a quadratic hypersurface in $\LL$, so $L[\LL\cap Z]$ is a quadratic hypersurface in $\R^k$, unless $L\given\LL$ is singular in which case $L[\LL]$ is a hyperplane in $\R^k$. Either way, the $\mu$-measure of neighborhoods of $\Psi^{-1}(\LL)\subset L[\LL\cap Z]$ can be bounded using Lemma \ref{lemma1}. Since $\Psi$ is bi-Lipschitz, this completes the proof.
\end{proof}

We record the following corollary of Lemma \ref{lemmasimplex} for use in the proof below:

\begin{corollary}
\label{corollarysimplex}
Let $Z\subset\R^d$ be a hypersurface, and let $K\subset Z_\reg$ be a compact set. Then there exists $\kappa > 0$ such that for every ball $B = B(\xx,\rho)\subset\R^d$, the set
\[
\left\{\frac\pp q\in \Q^d\cap K\cap B : q \leq \kappa/\rho\right\}
\]
is contained in an affine hyperplane $\LL = \LL_{B,Z}$.
\end{corollary}

\subsection{Finishing the proof using the method of \cite{PollingtonVelani}}
As in Theorem \ref{maintheorem}, let $\mu$ be a measure on an open set $U\subset\R^k$ which is absolutely $\alpha$-friendly and Ahlfors $\delta$-regular, let $\Psi:U\to Z_\reg$ be a local coordinate chart on a quadratic hypersurface $Z\subset\R^d$, let $\nu = \Psi[\mu]$, and fix $c\geq 1$. Let $K\subset Z_\reg$ be a compact neighborhood of $\Supp(\nu)$. Then we can apply both Corollary \ref{corollaryquadraticneighborhood} and Corollary \ref{corollarysimplex}.

We wish to show that the dimension bound \eqref{main} holds. To this end, after fixing $s > \delta - \alpha\left(1 - \frac1c\right)$, we must construct a cover $\CC$ of $E := W_{c,Z}\cap\Supp(\nu)$ satisfying the hypotheses of Lemma \ref{HClemma}. We will construct this cover as the union of several smaller collections of sets.

For each $n\in\N$, let $\rho_n := 2^{-n}$, let $S_n \subseteq \Supp(\nu)$ be a maximal $\rho_n$-separated set, and let $\SS_n := \{B(\xx,\rho_n) : \xx \in S_n\}$. For each $B = B(\xx, \rho_n) \in \SS_n$, let $\LL_\xx = \LL_{2B}$ be as in Corollary \ref{corollarysimplex}, let $T_\xx \subseteq B\cap \thickvar{\LL_\xx\cap Z}{\rho_n^c}\cap \Supp(\nu)$ be a maximal $\rho_n^c$-separated set, and let $\TT_\xx = \{B(\yy, \rho_n^c): \yy \in T_B\}$. Then let
\[
\CC_n := \displaystyle\bigcup_{\xx \in S_n} \TT_\xx
\]
and
\[
\CC := \bigcup_{n\in\N} \CC_n.
\]
We claim that this collection satisfies the hypotheses of Lemma \ref{HClemma}. First, fix $\zz \in E = W_{c,Z} \cap \Supp(\nu)$. Choose $c'\in (c,\omega_Z(\zz))$, and fix $\frac{\pp}{q} \in \Q^d \cap Z$ such that $\left\| \zz - \frac{\pp}{q}\right\| \leq q^{-c'}$. Let $n$ be minimal such that $q \leq \kappa \rho_n^{-1}$. Let $\xx \in S_n$ satisfy $\zz \in B = B(\xx,\rho_n)$ (such an $\xx$ must exist since $S_n$ is maximal).

If $q$ is sufficiently large, then $\left\| \zz - \frac{\pp}{q}\right\| \leq q^{-c'} \leq \rho_n$, so $\left\| \xx - \frac{\pp}{q}\right\| \leq 2\rho_n$. By Corollary \ref{corollarysimplex}, $\frac{\pp}{q} \in \LL_\xx$. On the other hand, the minimality of $n$ implies that $q\geq \kappa \rho_{n + 1}^{-1}$ and thus if $n$ is sufficiently large, then $q^{-c'} \leq (2/\kappa)^{c'} \rho_n^{c'} \leq \rho_n^c$. So
\[
\zz \in B\cap \thickvar{\LL_\xx \cap Z}{q^{-c'}} \subseteq B\cap \thickvar{\LL_\xx\cap Z}{\rho_n^c} \subseteq \bigcup\big(\TT_\xx\big) = \CC_n.
\]
It follows that $\zz$ lies in infinitely many $\CC_n$, namely one for each approximant $\frac{\pp}{q}$.

On the other hand, for all $\beta < \alpha$
\begin{align*}
\operatorname{cost}_s(\CC) &\leq_\pt \sum_{n\in\N} \sum_{\xx \in S_n} \sum_{\yy \in T_\xx} (2\rho_n^c)^s\\
&\asymp_\times \sum_n \sum_{\xx} \sum_\yy \nu\big(B(\yy,\rho_n^c)\big) (\rho_n^c)^{(s-\delta)} \note{Ahlfors $\delta$-regularity}\\
&\asymp_\times \sum_n \rho_n^{c(s-\delta)} \sum_\xx \nu\big(B(\xx,2\rho_n)\cap \thickvar{\LL_\xx \cap Z}{2\rho_n^c}\big) \note{bounded multiplicity}\\
&\lesssim_\times \sum_n \rho_n^{c(s-\delta)} \sum_\xx (\rho_n^{c-1})^\beta \nu\big( B(\xx,2\rho_n)\big) \note{Corollary \ref{corollaryquadraticneighborhood}}\\
&\asymp_\times \sum_n (\rho_n)^{c(s-\delta) + \beta(c-1)} \nu(Z) \note{bounded multiplicity}\\
&<_\pt \infty,
\end{align*}
where the last inequality holds assuming $c(s-\delta) + \beta(c-1) > 0$, i.e.
\begin{equation}
\label{sbetacomparison}
s > \delta - \beta\left(1 - \frac{1}{c}\right).
\end{equation}
For all $s > \delta - \alpha\left(1 - \frac1c\right)$, there exists $\beta < \alpha$ such that \eqref{sbetacomparison} holds, and thus $\HD(E)\leq s$. This completes the proof of \eqref{main}.

%

\bibliographystyle{amsplain}

\bibliography{bibliography}

\end{document}